\documentclass[12pt]{article}

\usepackage{graphicx}
\usepackage{geometry}
\usepackage[pdftex,
            colorlinks,
            pdftitle={Computing the invariant measure and the Lyapunov exponent for 
one-dimensional maps using a measure-preserving polynomial basis},
            pdfauthor={Philip J. Aston, Oliver Junge},
            hyperindex,
            filecolor=blue,
            urlcolor=blue,
            citecolor=blue,
            linkcolor=blue]{hyperref}
            
\usepackage{amsmath,amsthm,amssymb}            
            
\newtheorem{theorem}{Theorem}[section]
\newtheorem{lemma}[theorem]{Lemma}
\newtheorem{corollary}[theorem]{Corollary}

\theoremstyle{definition}

\theoremstyle{remark}
\newtheorem{remark}[theorem]{Remark}

\numberwithin{equation}{section}

\newcommand{\Ref}[1]{(\ref{#1})}
\newcommand{\m}{{\bf m}}
\newcommand{\R}{{\mathbb R}}

\begin{document}
%\addtolength{\baselineskip}{0.4mm}

\title
{Computing the invariant measure and the Lyapunov exponent for 
one-dimensional maps using a measure-preserving polynomial basis}

\author{Philip J. Aston\footnote{Department of Mathematics, University of Surrey, Guildford, Surrey GU2 7XH, UK, \texttt{P.Aston@surrey.ac.uk}}, 
Oliver Junge
\footnote{Technische Universit\"at M\"unchen, Zentrum Mathematik,
Boltzmannstr. 3, D-85747 Garching, Germany, 
\texttt{oj@tum.de}}}

%\subjclass[2010]{37M25,65P20}

\date{November, 2011}

\maketitle

\begin{abstract}
We consider a generalisation of Ulam's method for approximating
invariant densities of one-dimensional chaotic maps. Rather than use
piecewise constant polynomials to approximate the density, we use 
polynomials of degree $n$ which are defined by the requirement that 
they preserve the measure on $n+1$ neighbouring subintervals. Over the
whole interval, this results in a discontinuous piecewise polynomial
approximation to the density. We prove error results where this approach
is used to approximate smooth densities. We also consider the computation
of the Lyapunov exponent using the polynomial density and show that the
order of convergence is one order better than for the density itself.
Together with using cubic polynomials in the density approximation, this
yields a very efficient method for computing highly accurate estimates
of the Lyapunov exponent.  We illustrate the theoretical findings
with some examples.
\end{abstract}

\section{Introduction}

Suppose that a one-dimensional piecewise smooth map $g:I\to I$ on a 
compact interval $I\subset\R$ has a unique chaotic attractor with an
absolutely continuous invariant measure $\mu$. We consider piecewise
polynomial approximations to the density $d$ associated with $\mu$.  
To this end, we define a partition of the interval $I$ by dividing it into a 
number of equal subintervals $I_i=[x_{i-1},x_i]$, $i=1,2,\ldots,N$, each of 
length $h=(x_N-x_0)/N$. For each subinterval, we define
$$m_i=\int_{I_i}d\mu, \quad i=1,2,\ldots,N.$$
Thus, $m_i$ is the mass of the interval $I_i$ with respect to the invariant
measure $\mu$. Clearly we have the property that
$$\sum_{i=1}^Nm_i=1,$$
since $\mu$ is a probability measure. We want to approximate the density $d$ 
of $\mu$ given that the only information that we have about $\mu$ is the 
masses $m_i$, $i=1,2,\ldots,N$.

For approximating invariant densities, Ulam's classical method \cite{Ul60a}
consists in using piecewise constant
approximations on each subinterval with the constants 
being found by computing a fixed point (i.e.\ an eigenvector at the
eigenvalue $1$) of a (discretized) Markov operator, the so called
\emph{Frobenius-Perron operator} or more generally \emph{transfer
operator}.  This operator describes how measures (and densities) are
pushed forward under the dynamics of the map $g$.
Ulam's method has been proved to be convergent for piecewise expanding maps 
\cite{Li76a}. Since then much effort has gone into analysing the method 
\cite{BoMu01a,DeFrSe00a,MR1278389,DiLi91a,DiLi98a,DiZh02a}, computational
questions \cite{MR2319207,Hu94a}, extending the proof to more general 
classes of transformations 
\cite{DiLiZh02a,Fr96a,Fr98a,Fr99a,Hu96a,Hu98a,HuMi92a,Mu97a} or extending 
the method to higher dimensional systems
\cite{DeHoJu97a,DeJu98a,DeJu99a,DiZh96a,Fr95a,Ju01a,JuKo08a}.

For non-smooth invariant densities, a piecewise constant approximation may
be a good choice.  However, in the smooth case (for example if the map is
perturbed by white noise), it would be better to use a higher order
approximation.  Higher order polynomial approximations to the invariant
density have  been  obtained by Ding and coworkers 
\cite{DiDuLi93a,DiLi98a,DiRh04a,DiZh99a}. They used a Galerkin method for
finding a piecewise polynomial approximation to the density but did not go
on to consider the computation of the Lyapunov exponent.

Increasingly, numerical
methods are being developed which incorporate and preserve features of the 
problem being considered. One well-known example of this is geometric 
integrators for ODE's which encode properties of the ODE's, such as a 
symplectic structure or conservation of an invariant, into the numerical 
method \cite{HaLuWa02}. Our approach to the problem of computing an
invariant density is based on this philosophy. The essence of the density
is that the measure should be preserved on any subinterval and so we 
construct a discontinuous piecewise polynomial basis which is defined 
locally by the 
requirement that the measure is preserved on neighbouring intervals.
Combining many such polynomials over the whole interval, the coefficients
of the approximate invariant density with respect to this basis can again
be found by computing a fixed point of the associated discrete
Frobenius-Perron operator.  Using this new basis, which turns out to lead
to a \emph{Petrov-Galerkin} discretisation of the Frobenius-Perron
operator, we improve upon the convergence rates in \cite{DiDuLi93a,DiLi98a}
for piecewise quadratic approximations.  In addition, we consider the
computation of the Lyapunov exponent using the invariant density and give
corresponding error estimates. It turns out (see Theorem \ref{lethm}) that
using our measure preserving basis, one actually gains one order in the
convergence of the Lyapunov exponent.  Using cubic polynomials, this leads
to a highly efficient method for computing estimates of the Lyapunov
exponent.

An outline of the paper is as follows: In Section 2, we derive the
measure-preserving polynomial basis and perform an error analysis for a
density that is approximated by such a
polynomial. Section 3 considers the problem of integration when the density
is replaced by a polynomial while Section 4 gives an error analysis when
the density is found as a solution of the discretised Frobenius-Perron
fixed point equation. The computation of the Lyapunov exponent using the 
polynomial
approximation to the density is considered in Section 5, and the errors in
the Lyapunov exponent are derived. These results are illustrated with some
examples. Finally, in Section 6, two extensions of this work are described.

\section{Measure-Preserving Polynomial Approximations to the Density}
\label{poly-sec}

Our aim is to construct a piecewise polynomial approximation to an invariant
density. To do this, we start by considering the problem of obtaining a 
polynomial approximation to the
density function over a subset of $n+1$ adjacent intervals
$[x_i,x_{i+n+1}]=\cup_{j=0}^nI_{i+j+1}$ for some $i$. As with
Gaussian quadrature, we prefer to work with the standard interval $[-1,1]$
and so we perform a change of variables to map the interval
$[x_i,x_{i+n+1}]$ onto $[-1,1]$ using
\begin{equation}\label{tofx}
t=t(x)=-1+\frac{2(x-x_i)}{(n+1)h}.
\end{equation}
The corresponding inverse transformation is given by
\begin{equation}\label{xoft}
x(t)=\frac{1}{2}(n+1)h(t+1)+x_i.
\end{equation}
We transform the points $x_{i+j}$ onto the interval $[-1,1]$ by
\begin{equation}\label{tj}
t_j=t(x_{i+j})=-1+\frac{2j}{n+1},\quad j=0,1,\ldots,n+1,
\end{equation}
assuming that the points are evenly spaced. We also define a new
density function $D(t)$ which satisfies
$$D(t)~dt=d(x)~dx,$$
which implies that
\begin{equation}\label{Dt}
D(t)=\frac{dx}{dt}d(x(t))=\frac{1}{2}(n+1)h d(x(t)),
\end{equation}
using \Ref{xoft}. Using this definition, we find that
\begin{equation}\label{newint}
\int_{x_i}^{x_{i+n+1}}f(x)d(x)~dx=\int_{-1}^1F(t)D(t)~dt,
\end{equation}
where $F(t)=f(x(t))$. Defining 
\begin{equation}\label{Mj}
M_j=m_{i+j+1},\quad j=0,\ldots,n,
\end{equation}
we obtain
$$M_j=m_{i+j+1}=\int_{x_{i+j}}^{x_{i+j+1}}d(x)~dx=\int_{t_j}^{t_{j+1}}D(t)~dt,
\quad j=0,\ldots,n.$$
We now use the $n+1$ values $M_j$, $j=0,\ldots,n$ to obtain a
polynomial $p_n(t)\in\Pi_n$ which approximates the density
function $D(t)$ which can be used in the integration formula
\Ref{newint} with a higher order Gaussian quadrature method to obtain
higher accuracy for the integral. The criterion which we use to define
the polynomial is that it should preserve the measure on each subinterval,
that is
\begin{equation}\label{condition}
\int_{t_j}^{t_{j+1}}p_n(t)~dt=M_j,\quad j=0,1,\ldots,n.
\end{equation}
We will consider two different methods for determining the
measure-preserving polynomial $p_n(t)$, but we first establish uniqueness.

\begin{lemma}
The polynomial $p_n(t)\in\Pi_n$ satisfying \Ref{condition} is unique.
\end{lemma}
\begin{proof}
Assume that $p_n(t)$ and $q_n(t)$ are both polynomials that satisfy the
conditions \Ref{condition} and define
$$h_n(t)=p_n(t)-q_n(t).$$
Then $h_n(t)\in\Pi_n$ and satisfies the conditions
$$\int_{t_j}^{t_{j+1}}h_n(t)~dt=0,\quad j=0,1,\ldots,n.$$
These conditions imply that, for each value of $j$, there is at least one
point $\tau_j\in(t_j,t_{j+1})$ such that $h_n(\tau_j)=0$. Thus, in total, there 
are at least $n+1$ distinct points in the interval $(-1,1)$ at which $h_n$ 
vanishes. The only polynomial of degree $n$ that has at least $n+1$ distinct
zeros is the zero polynomial, and so $h_n(t)=0$, giving $p_n(t)=q_n(t)$.
\end{proof}

There are two approaches which can be used to determine the
polynomial $p_n(t)$, both of which will be useful later.

\subsection{Constructing the measure-preserving polynomial by interpolating
the measure}

We define the measure $\nu$ associated with the density function $D(t)$ by
\begin{equation}\label{cumden}
\nu(t)=\int_{-1}^tD(\tau)d\tau,
\end{equation}
and we note that
\begin{equation}\label{nuvalues}
\nu(t_j)=\sum_{k=0}^{j-1}M_k,\quad j=1,2,\ldots,n+1,\quad\nu(t_0)=0.
\end{equation}
Strictly speaking, $\nu$ is not a probability measure since
$\int_{-1}^1d\nu\neq 1$, but it is a part of the probability measure $\mu$.
From \Ref{nuvalues}, we have $n+2$ values of the function 
$\nu(t)$ which we can
interpolate using a polynomial $q_{n+1}(t)\in\Pi_{n+1}$.
The derivative of this polynomial will give an approximation to
the density function which satisfies \Ref{condition}. Thus, we
define
$$p_n(t)=q_{n+1}'(t),$$
where prime denotes differentiation with respect to $t$. Using the
Lagrange form of the interpolating polynomial \cite{BF10}, we thus have that
\begin{eqnarray*}
q_{n+1}(t)&=&\sum_{j=1}^{n+1}\nu(t_j)L_{n,j}(t)=\sum_{j=1}^{n+1}\sum_{k=0}^{j-1}M_k L_{n,j}(t)=\sum_{k=0}^n M_k\sum_{j=k+1}^{n+1} L_{n,j}(t)
\end{eqnarray*}
where
$$L_{n,j}(t)=\prod_{\stackrel{i=0}{i\neq j}}^n
\left(\frac{t-t_i}{t_j-t_i}\right).$$
Note that there is no term in the definition of $q_{n+1}(t)$ with
$j=0$ since $\nu(t_0)=0$. Thus, one representation of the polynomial
which satisfies the conditions \Ref{condition} is
$$p_n(t)=\sum_{k=0}^n M_k\sum_{j=k+1}^{n+1} L_{n,j}'(t).$$

\subsection{Constructing the measure-preserving polynomial using an 
appropriate basis}

The derivatives of the basis functions $L_{n,j}(t)$ in the previous
method are clearly quite messy to evaluate. An alternative approach is
to seek the polynomial $p_n(t)$ in the form
\begin{equation}\label{poly}
p_n(t)=\sum_{k=0}^n M_k\ell_{n,k}(t).
\end{equation}
The basis functions $\ell_{n,k}(t)$ can be found by substituting this
form of the polynomial into the conditions \Ref{condition} and
equating coefficients of $M_k$. This implies that for each
$k=0,1,\ldots,n$, we have the $n+1$ conditions
\begin{equation}\label{lconds}
\int_{t_j}^{t_{j+1}}\ell_{n,k}(t)~dt=\left\{\begin{array}{ll}
0, & k\neq j \\[2mm] 1, & k=j \end{array}\right.\hspace{5mm}j=0,1,\ldots,n.
\end{equation}
These equations represent a system of linear equations for the
$n+1$ coefficients in the polynomial $\ell_{n,k}(t)$.

Of course this is only a different representation of the same
polynomial found using the measure and so
$$\ell_{n,k}(t)=\sum_{j=k+1}^{n+1} L_{n,j}'(t).$$
The basis functions defined by the conditions \Ref{lconds} for some low 
values of $n$ are listed in Table
\ref{basisfns}. The three basis functions for the case $n=2$ are shown in
Fig.\ \ref{l2}. Some properties of these basis functions can be determined.

\begin{table}[t]
\begin{center}
\renewcommand{\arraystretch}{1.5}
\begin{tabular}{|c|l|}
\hline
\hspace{1mm}$n=0$\hspace{1mm} &~ $\ell_{0,0}(t)=\frac{1}{2}$ \\
\hline
$n=1$ &~ $\ell_{1,0}(t)=-t+\frac{1}{2}$ \\
      &~ $\ell_{1,1}(t)=\ell_{1,0}(-t)$ \\
\hline
$n=2$ &~ $\ell_{2,0}(t)=\frac{1}{16}(27t^2-18t-1)$\\
      &~ $\ell_{2,1}(t)=\frac{1}{8}(-27t^2+13)$\\
      &~ $\ell_{2,2}(t)=\ell_{2,0}(-t)$ \\
\hline
$n=3$ &~ $\ell_{3,0}(t)=\frac{1}{6}(-16t^3+12t^2+2t-1)$\\
      &~ $\ell_{3,1}(t)=\frac{1}{6}(48t^3-12t^2-30t+7)$\\
      &~ $\ell_{3,2}(t)=\ell_{3,1}(-t)$\\
      &~ $\ell_{3,3}(t)=\ell_{3,0}(-t)$~~~\\
\hline
\end{tabular}
\end{center}
\hspace{5mm}

\caption{Basis functions for the measure-preserving polynomial approximation 
to the density.}
\label{basisfns}
\end{table}

\begin{figure}
\centering
\includegraphics[width=0.5\textwidth]{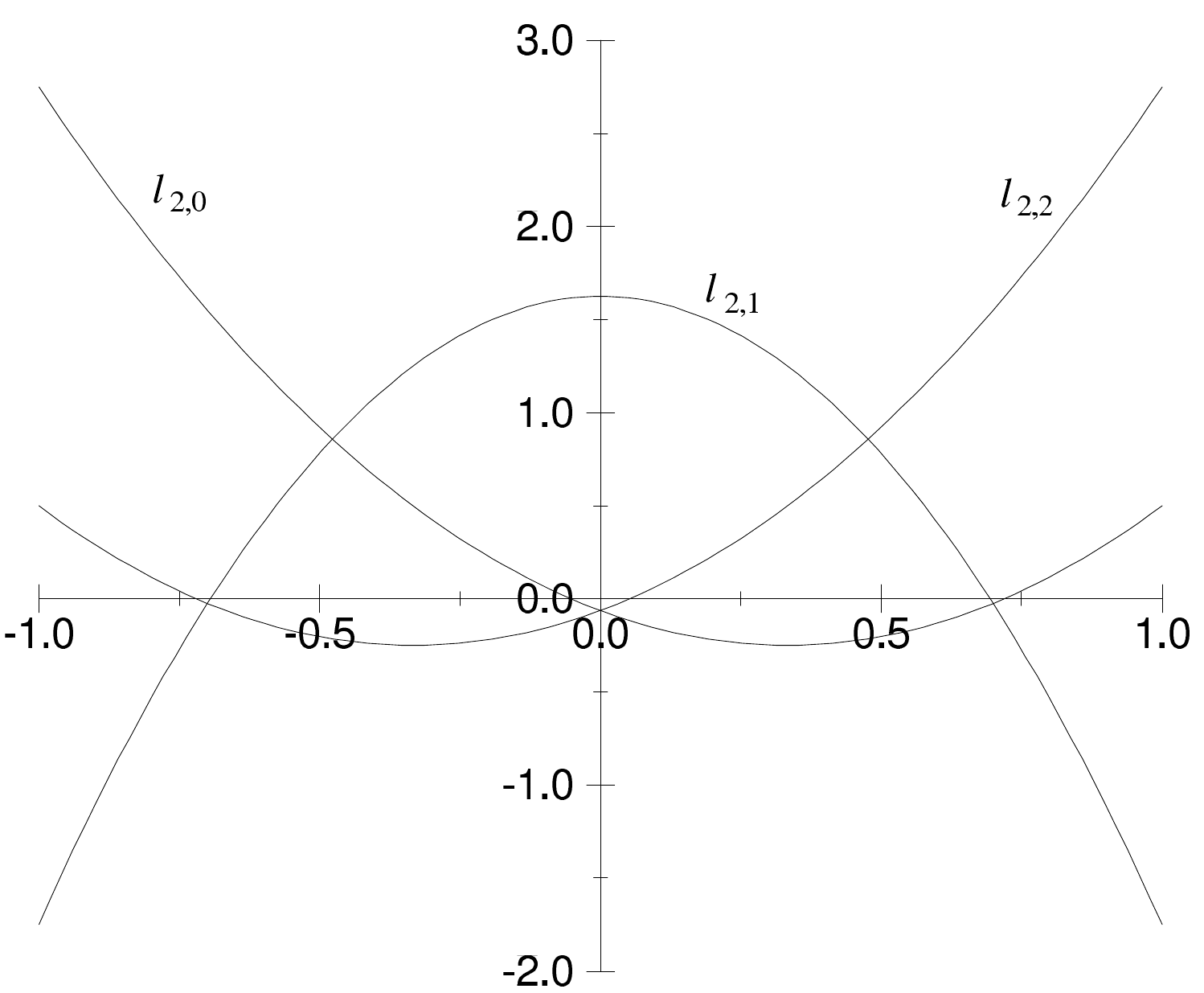}
\caption{The three basis functions for $n=2$.}
\label{l2}
\end{figure}

\newcounter{c1}
\begin{lemma}
The basis functions $\ell_{n,k}(t)$ have the following properties:
\begin{list}{\rm(\roman{c1})}{\usecounter{c1}}
\item
$\ell_{n,n-k}(t)=\ell_{n,k}(-t)$;
\item
$\displaystyle\int_{-1}^1\ell_{n,k}(t)~dt=1$;
\item
$\displaystyle\sum_{k=0}^n\ell_{n,k}(t)=\frac{n+1}{2}$;
\item
The function $\ell_{n,k}(t)$ has $n$ simple roots with one root in
each of the intervals $(t_j,t_{j+1})$ for $j\neq k$.
\end{list}
\end{lemma}

\begin{proof}
Results (i) and (ii) follow directly from the definition of the basis
functions. For the third result, consider the case that $D(t)=c$ for
some constant $c$. Then $M_i=2c/(n+1)$, $i=0,\ldots,n$ and
$p_n(t)=D(t)=c$. Substituting
these into \Ref{poly} gives the required result. For the fourth result,
note that the condition that the integral of $\ell_{n,k}(t)$ is zero
over the subinterval $[t_j,t_{j+1}]$
implies that there must be at least one point in $(t_j,t_{j+1})$ at which
the function is zero by the first Mean Value Theorem for Integrals.
Since one root in each interval gives us the
maximum possible number of roots, then there must be precisely one root in
each interval.
\end{proof}

\subsection{Error analysis}

Having determined methods for finding a polynomial approximation
$p_n(t)$ to the density function $D(t)$, we must now consider the
errors in this approximation. Thus, we write
\begin{equation}\label{error}
D(t)=p_n(t)+e_n(t),
\end{equation}
for some error function $e_n(t)$.

Using the first approach of interpolating the measure, the error in 
the interpolating polynomial $q_{n+1}(t)$ is given by
$$\nu(t)=q_{n+1}(t)+\varepsilon_{n+1}(t),$$
where
\begin{equation}\label{epserror}
\varepsilon_{n+1}(t)=\frac{\nu^{(n+2)}(\kappa(t))}{(n+2)!}\prod_{j=0}^{n+1}
(t-t_j),
\end{equation}
for some $\kappa(t)\in(-1,1)$, assuming that $\nu(t)\in C^{n+2}[-1,1]$
(see \cite{BF10}). However, to obtain the error in the
density requires differentiating $\varepsilon_{n+1}(t)$ which does not
lead to a nice form for $e_n(t)$. Thus, we also derive the error based on
the alternative approach of determining the polynomial in terms of the
basis functions $\ell_{n,k}(t)$.

Integrating \Ref{error} over one subinterval and using \Ref{condition}
gives that
\begin{equation}\label{inten}
\int_{t_j}^{t_{j+1}}e_n(t)~dt=0,\quad j=0,1,\ldots,n.
\end{equation}
By the first Mean Value Theorem for Integrals, this implies that there 
is at least one point $\tau_j\in(t_j,t_{j+1})$ such that 
$e_n(\tau_j)=0$, $j=0,1,\ldots,n$ and so there are at least $n+1$ distinct
points in the interval $(-1,1)$ at
which the polynomial interpolates the density. If we choose precisely one
such point in each subinterval, then this interpolation problem defines a
unique polynomial in $\Pi_n$ which must be the polynomial
$p_n(t)$. Thus, regarding the problem as an interpolation problem,
we can express the error function as
\begin{equation}\label{ent}
e_n(t)=\frac{D^{(n+1)}(\xi(t))}{(n+1)!}\prod_{j=0}^n(t-\tau_j),
\end{equation}
where $\xi(t)\in(-1,1)$ assuming that $D(t)\in C^{n+1}[-1,1]$ (see
\cite{BF10}). Note that the points $\tau_j$ depend on the particular
function $D(t)$.

Since there is a simple relationship \Ref{cumden} between $\nu(t)$ and
$D(t)$, there must also be a similar relationship between the errors
which is given by
\begin{equation}\label{errors}
\varepsilon_{n+1}(t)=\int_{-1}^te_n(\tau)d\tau.
\end{equation}

\section{Integration}\label{int-sec}

There are two different types of integral that will be needed later
which involve the density and so we consider them both at this stage,
using the polynomial approximation to the density, and also provide an error
analysis.

\subsection{Case I}

When computing the density using the Frobenius-Perron equation,
integrals of the form
$$\int_{x_\ell}^{x_r}d(x)~dx=\int_{t_\ell}^{t_r}D(t)~dt$$
will be required, where $t_\ell=t(x_\ell)$, $t_r=t(x_r)$ and 
$[t_\ell,t_r]\subseteq[-1,1]$. We assume in this section that the measures
$m_i$ on the intervals $I_i$, $i=1,\ldots,N$ are known exactly and we
consider only the errors associated with the piecewise polynomial
approximation to the measure. Using the polynomial approximation to
the density then gives the following result.

\begin{theorem}\label{case1thm}
If $D(t)\in C^{n+1}[-1,1]$, then
\begin{equation}\label{res1}
\int_{t_\ell}^{t_r}D(t)~dt=\int_{t_\ell}^{t_r}p_n(t)~dt
+O(h^{n+2}),
\end{equation}
\begin{equation}\label{res2}
\int_{t_\ell}^{t_r}|D(t)-p_n(t)|~dt=O(h^{n+2}).
\end{equation}
\end{theorem}

\begin{proof}
Integrating the
polynomial approximation with error given by \Ref{error}, we obtain
$$\int_{t_\ell}^{t_r}D(t)~dt=\int_{t_\ell}^{t_r}p_n(t)~dt+
\int_{t_\ell}^{t_r}e_n(t)~dt.$$
Thus, for \Ref{res1}, we need to estimate
\begin{equation}\label{int1}
\int_{t_\ell}^{t_r}e_n(t)~dt.
\end{equation}
Similarly, for \Ref{res2}, we need to estimate
\begin{equation}\label{int2}
\int_{t_\ell}^{t_r}|e_n(t)|~dt.
\end{equation}
From \Ref{ent}, we see that the function $e_n(t)$ involves $D^{(n+1)}(\xi(t))$.
Now from \Ref{Dt}, we know that
$$D(t)=chd(x(t)),$$
where $c=\frac{1}{2}(n+1)$. Since $dx/dt=ch$, we therefore note
that
$$D^{(n+1)}(t)=(ch)^{n+2}d^{(n+1)}(x(t)),$$
and so
$$D^{(n+1)}(\xi(t))=O(h^{n+2}).$$
Thus, the integrals \Ref{int1} and \Ref{int2} are both $O(h^{n+2})$ as
required.
\end{proof}

Using \Ref{res1}, we now obtain
\begin{eqnarray*}
\int_{x_\ell}^{x_r}d(x)~dx&=&\int_{t_\ell}^{t_r}D(t)~dt\\
&=&\int_{t_\ell}^{t_r}p_n(t)~dt+O(h^{n+2})\\
&=&\int_{t_\ell}^{t_r}\sum_{k=0}^nM_k\ell_{n,k}(t)~dt+O(h^{n+2})\\
&=&\sum_{k=0}^nm_{i+k+1}\int_{t_\ell}^{t_r}\ell_{n,k}(t)~dt+O(h^{n+2}),
\end{eqnarray*}
where we used \Ref{Mj} in the final step.
Clearly, in practice we simply drop the error term when computing this
integral.

\subsection{Case II}\label{case2}

The second type of integral that we will need to evaluate is given by
\Ref{newint}. As in the previous case, we again assume that the measures 
$m_i$ on the intervals $I_i$, $i=1,\ldots,N$ are known exactly.
Then we have that
\begin{equation}\label{integral2}
\int_{-1}^1F(t)D(t)~dt=\int_{-1}^1F(t)p_n(t)~dt+\int_{-1}^1F(t)e_n(t)~dt.
\end{equation}
We first consider the error term in \Ref{integral2}. For later use, we
define the function
$$\phi_{n+2}(t)=\prod_{j=0}^{n+1}(t-t_j),$$
where the evenly spaced points $t_j$ are given by \Ref{tj}.

\begin{theorem}\label{int2-thm}
If $n$ is odd, $F\in C^2[-1,1]$ and $D\in C^{n+2}[-1,1]$
then the error term in \Ref{integral2} is
\begin{equation}\label{nodd}
\int_{-1}^1F(t)e_n(t)~dt=-\left[\frac{F''(\alpha)D^{(n+1)}(\delta)}{(n+2)!}
+\frac{F'(\alpha)D^{(n+2)}(\gamma)}{(n+3)!}\right]
\int_{-1}^1 t\phi_{n+2}(t)~dt,
\end{equation}
where $\alpha,\delta,\gamma\in(-1,1)$.

If $n$ is even, $F\in C^1[-1,1]$ and $D\in C^{n+1}[-1,1]$ then the error
term in \Ref{integral2} is
\begin{eqnarray}
\int_{-1}^1 F(t)e_n(t)~dt&=&-
\left[\frac{F''(\alpha)D^{(n+1)}(\delta)}{(n+2)!}
+\frac{F'(\alpha)D^{(n+2)}(\gamma)}{(n+3)!}\right]\times\nonumber\\
&&(\alpha-1)\int_{-1}^{t_n}\phi_{n+2}(t)~dt\nonumber\\
&&-\frac{F'(\xi)D^{(n+1)}(\eta)}{(n+2)!}
\int_{-1}^{1}\phi_{n+2}(t)~dt,\label{neven}
\end{eqnarray}
where $\alpha,\delta,\gamma,\xi,\eta\in(-1,1)$.
\end{theorem}

\begin{proof}
We analyse the error by integrating by parts and using the error in the 
approximation to the measure $\nu(t)$ given by \Ref{epserror}. Integrating by
parts and using \Ref{errors} gives
$$\int_{-1}^1F(t)e_n(t)~dt=[F(t)\varepsilon_{n+1}(t)]_{-1}^1
-\int_{-1}^1F'(t)\varepsilon_{n+1}(t)~dt.$$
The boundary terms vanish since $\varepsilon_{n+1}(-1)=0$ by
\Ref{errors} and $\varepsilon_{n+1}(1)=0$ by \Ref{errors} and \Ref{inten}.
Thus,
\begin{equation}\label{first-step}
\int_{-1}^1F(t)e_n(t)~dt=-\frac{1}{(n+2)!}\int_{-1}^1F'(t)
\nu^{(n+2)}(\kappa(t))\phi_{n+2}(t)~dt.
\end{equation}
This error term is very similar to that for simple Newton-Cotes integration
formulae \cite{RR78} since the $t_j$'s are evenly spaced points, except for
the extra term $F'(t)$ which complicates the analysis.

We first consider the case when $n$ is odd. In this case, $\phi_{n+2}(t)$ 
is an odd function since the points $t_j$ are evenly spaced, and
this results in an error which is of higher order than would be
expected, as in the case of Newton-Cotes formulae. 

We define
$$\psi_{n+3}(t)=\int_{-1}^t\phi_{n+2}(\tau)~d\tau.$$
Since $\phi_{n+2}(t)$ is an odd function, it follows that $\psi_{n+3}(t)$ is
of one sign \cite{RR78}, and that
\begin{equation}\label{psi}
\psi_{n+3}(1)=0.
\end{equation}
Integrating \Ref{first-step} by parts, the boundary terms again disappear 
using \Ref{psi} and so
$$\int_{-1}^1F'(t)\nu^{(n+2)}(\kappa(t))\phi_{n+2}(t)~dt
=-\int_{-1}^1\frac{d}{dt}\left(
F'(t)\nu^{(n+2)}(\kappa(t))\right)\psi_{n+3}(t)~dt.$$
Since $\psi_{n+3}(t)$ is of one sign, the Mean Value Theorem for
Integrals can be used, giving
\begin{eqnarray}
\int_{-1}^1F'(t)\nu^{(n+2)}(\kappa(t))\phi_{n+2}(t)~dt
&=&-\frac{d}{dt}\left.\left(F'(t)\nu^{(n+2)}
(\kappa(t))\right)\right|_{t=\alpha}\int_{-1}^1\psi_{n+3}(t)~dt\nonumber\\
&=&\frac{d}{dt}\left.\left(F'(t)\nu^{(n+2)}(\kappa(t))\right)
\right|_{t=\alpha} \int_{-1}^1 t\phi_{n+2}(t)~dt,\label{noddresult}
\end{eqnarray}
where $\alpha\in(-1,1)$ and the second step involved another integration by 
parts.

It can be shown \cite{R63} that
\begin{equation}\label{int-result}
\frac{d}{dt}\nu^{(n+2)}(\kappa(t))=\frac{\nu^{(n+3)}(\beta(t))}{n+3},
\end{equation}
for some $\beta(t)\in(-1,1)$. Thus, substituting \Ref{noddresult} into 
\Ref{first-step} and using \Ref{int-result}, we obtain
$$\int_{-1}^1F(t)e_n(t)~dt=-\left[\frac{F''(\alpha)\nu^{(n+2)}(\delta)}{(n+2)!}
+\frac{F'(\alpha)\nu^{(n+3)}(\gamma)}{(n+3)!}\right]
\int_{-1}^1 t\phi_{n+2}(t)~dt,$$
where $\delta=\kappa(\alpha)$ and $\gamma=\beta(\alpha)$. Now
$\nu'(t)=D(t)$ and so substituting for $\nu$ gives \Ref{nodd} as claimed.

When $n$ is even, the proof is more complicated, as it is for the standard
Newton-Cotes result.
We first break the integral in \Ref{first-step} into two, giving
\begin{eqnarray}
\int_{-1}^1F'(t)\nu^{(n+2)}(\kappa(t))\phi_{n+2}(t)~dt&=&
\int_{-1}^{t_n}F'(t)\nu^{(n+2)}(\kappa(t))\phi_{n+2}(t)~dt\nonumber\\
&&+\int_{t_n}^1F'(t)\nu^{(n+2)}(\kappa(t))\phi_{n+2}(t)~dt.\label{two-ints}
\end{eqnarray}
For the second integral, $\phi_{n+2}(t)$ is of one sign on the interval
$[t_n,1]$ and so the Mean Value Theorem for Integrals gives that
\begin{equation}\label{second-int}
\int_{t_n}^1F'(t)\nu^{(n+2)}(\kappa(t))\phi_{n+2}(t)~dt=
F'(\theta)\nu^{(n+2)}(\kappa(\theta))\int_{t_n}^1\phi_{n+2}(t)~dt,
\end{equation}
for some $\theta\in(t_n,1)$.

We rewrite the first integral in \Ref{two-ints} as
$$\int_{-1}^{t_n}F'(t)\nu^{(n+2)}(\kappa(t))\phi_{n+2}(t)~dt
=\int_{-1}^{t_n}F'(t)\nu^{(n+2)}(\kappa(t))(t-1)\phi_{n+1}(t)~dt,$$
and we note that $\phi_{n+1}(t)$ is an odd function about the midpoint of
the interval $[-1,t_n]$. Thus, using the result \Ref{noddresult} for $n$ 
odd then gives
\begin{eqnarray}
\int_{-1}^{t_n}F'(t)\nu^{(n+2)}(\kappa(t))\phi_{n+2}(t)~dt
&=&\frac{d}{dt}\left.\left\{F'(t)\nu^{(n+2)}(\kappa(t))(t-1)\right\}
\right|_{t=\alpha}\times\nonumber\\
&&\int_{-1}^{t_n}t\phi_{n+1}(t)~dt\nonumber\\
&=&\Bigg[\frac{d}{dt}\left.\left\{F'(t)\nu^{(n+2)}(\kappa(t))\right\}
\right|_{t=\alpha}(\alpha-1) \nonumber\\
&&~~+F'(\alpha)\nu^{(n+2)}(\kappa(\alpha))
\Bigg]\int_{-1}^{t_n}t\phi_{n+1}(t)~dt,\label{first-int}
\end{eqnarray}
where $\alpha\in(-1,t_n)$. Now since $\phi_{n+1}(t)$ is odd about the
midpoint of the interval,
\begin{equation}\label{phiodd}
\int_{-1}^{t_n}t\phi_{n+1}(t)~dt=\int_{-1}^{t_n}(t-1)\phi_{n+1}(t)~dt
=\int_{-1}^{t_n}\phi_{n+2}(t)~dt.
\end{equation}
Combining the integrals \Ref{second-int} and \Ref{first-int}, 
and using \Ref{phiodd}, we therefore obtain
\begin{eqnarray*}
\int_{-1}^1F'(t)\nu^{(n+2)}(\kappa(t))\phi_{n+2}(t)~dt&=&
\Bigg[\frac{d}{dt}\left.\left\{F'(t)\nu^{(n+2)}(\kappa(t))\right\}
\right|_{t=\alpha}(\alpha-1)\\
&&~~+F'(\alpha)\nu^{(n+2)}(\kappa(\alpha))
\Bigg]\int_{-1}^{t_n}\phi_{n+2}(t)~dt\\
&&+F'(\theta)\nu^{(n+2)}(\kappa(\theta))\int_{t_n}^1\phi_{n+2}(t)~dt.
\end{eqnarray*}
Now the two integrals have the same sign \cite{S50} and so
\begin{eqnarray*}
\int_{-1}^1F'(t)\nu^{(n+2)}(\kappa(t))\phi_{n+2}(t)~dt&=&
\frac{d}{dt}\left.\left\{F'(t)\nu^{(n+2)}(\kappa(t))\right\}
\right|_{t=\alpha}\times\\
&&(\alpha-1)\int_{-1}^{t_n}\phi_{n+2}(t)~dt\\
&&+F'(\xi)\nu^{(n+2)}(\kappa(\xi))\times\\
&&\left[\int_{-1}^{t_n}\phi_{n+2}(t)~dt
+\int_{t_n}^1\phi_{n+2}(t)~dt\right]\\
&=&\frac{d}{dt}\left.\left\{F'(t)\nu^{(n+2)}(\kappa(t))\right\}
\right|_{t=\alpha}\times\\
&&(\alpha-1)\int_{-1}^{t_n}\phi_{n+2}(t)~dt\\
&&+F'(\xi)\nu^{(n+2)}(\kappa(\xi))\int_{-1}^{1}\phi_{n+2}(t)~dt,
\end{eqnarray*}
for some $\xi\in(-1,1)$ (and satisfying $\alpha\leq\xi\leq\theta$). 
Expanding the derivative term and using \Ref{int-result} gives
\begin{eqnarray*}
\int_{-1}^1F'(t)\nu^{(n+2)}(\kappa(t))\phi_{n+2}(t)~dt&=&
\left[F''(\alpha)\nu^{(n+2)}(\kappa(\alpha))
+\frac{F'(\alpha)\nu^{(n+3)}(\beta(\alpha))}{n+3}\right]\\
&&\times(\alpha-1)\int_{-1}^{t_n}\phi_{n+2}(t)~dt\\ 
&&+F'(\xi)\nu^{(n+2)}(\kappa(\xi))\int_{-1}^{1}\phi_{n+2}(t)~dt,
\end{eqnarray*}
where $\beta(\alpha)\in(-1,1)$.
Substituting back into \Ref{first-step} then gives
\begin{eqnarray*}
\int_{-1}^1 F(t)e_n(t)~dt&=&-
\left[\frac{F''(\alpha)\nu^{(n+2)}(\kappa(\alpha))}{(n+2)!}
+\frac{F'(\alpha)\nu^{(n+3)}(\beta(\alpha))}{(n+3)!}\right]\times\\
&&(\alpha-1)\int_{-1}^{t_n}\phi_{n+2}(t)~dt\\
&&-\frac{F'(\xi)\nu^{(n+2)}(\kappa(\xi))}{(n+2)!}
\int_{-1}^{1}\phi_{n+2}(t)~dt,
\end{eqnarray*}
Finally, substituting for $\nu$ using $\nu'(t)=D(t)$ gives the stated
result, with $\delta=\kappa(\alpha)$, $\gamma=\beta(\alpha)$ and
$\eta=\kappa(\xi)$.
\end{proof}

These results can be extended to the case of evaluating an integral over
the whole interval $I$.

\begin{theorem}\label{thm2}
Let $p_N^n(x)$ be the discontinuous piecewise polynomial approximation to the
density, assuming that the measures $m_i$ on the intervals $I_i$,
$i=1,\ldots,N$ are known exactly.

If $n$ is odd, $f\in C^2(I)$, $d\in C^{n+2}(I)$ and $f(x)$ is not
a constant, then 
\begin{equation}\label{err1}
\int_I f(x)d(x)~dx=\int_I f(x)p_N^n(x)~dx+O(h^{n+3}).
\end{equation}

If $n$ is even, $f\in C^1(I)$, $d\in C^{n+1}(I)$ and $f(x)$ is not a
constant, then 
\begin{equation}\label{err2}
\int_I f(x)d(x)~dx=\int_I f(x)p_N^n(x)~dx+O(h^{n+2}).
\end{equation}
\end{theorem}

\begin{proof}
We first note that if $f(x)$ is a constant, then the error term in both
cases is zero by construction, and hence this case has been excluded.
The results of Theorem \ref{int2-thm} can be converted back to the 
original $x$ coordinates in order to obtain error terms as a power of $h$.
As in the proof of Theorem \ref{case1thm}, we note that $D(t)=chd(x(t))$ 
and that $dx/dt=ch$,
where $c=\frac{1}{2}(n+1)$. Using these results, we obtain, for $n$ odd,
\begin{eqnarray*}
\int_{x_i}^{x_{i+n+1}}f(x)d(x)~dx&=&\int_{-1}^1F(t)D(t)~dt\\
&=&\int_{-1}^1F(t)p_n(t)~dt+\int_{-1}^1F(t)e_n(t)~dt\\
&=&\int_{-1}^1F(t)p_n(t)~dt+O(h^{n+4}).
\end{eqnarray*}
Similarly, for $n$ even, we have
$$\int_{x_i}^{x_{i+n+1}}f(x)d(x)~dx
=\int_{-1}^1F(t)p_n(t)~dt+O(h^{n+3}).$$
We note that in the first term of \Ref{neven}, the term $\alpha-1$ can
be expressed, using \Ref{tofx}, as
$$\alpha-1=-2+\frac{2(\tilde\alpha-x_i)}{(n+1)h}=\frac{2(\tilde\alpha
-x_{i+n+1})}{(n+1)h}=O(1/h),$$
where $\tilde\alpha=x(\alpha)\in(x_i,x_{i+n+1})$, and so both terms in 
\Ref{neven} are $O(h^{n+3})$.

Integrating over the whole interval $I$ requires the sum of $N/(n+1)$ such
integrals. As usual, the error term for such a composite quadrature rule 
is one power of $h$ less than for the simple rule as it involves the sum of
the errors over the $N/(n+1)$ intervals, and $N=O(h^{-1})$. This gives the
stated results.
\end{proof}

\begin{theorem}\label{thm3}
The results of Theorem \ref{thm2} hold if $f\in C^0(I)$ but is piecewise
$C^2$ ($n$ odd) or piecewise $C^1$ ($n$ even).
\end{theorem}

\begin{proof}
Clearly for intervals $I_i$ which do not contain a discontinuity in the
derivative of $f(x)$, Theorem \ref{int2-thm} still holds. Thus, we need
consider only those intervals where there is point of discontinuity of the
derivative. In particular, we will consider only the case where there is a
single point of discontinuity in any particular interval, but the results
can easily be generalised to the case of multiple points of discontinuity.
Thus, we assume that on some interval $I_i$, the function $F(t)=f(x(t))$ is
given by
$$F(t)=\left\{\begin{array}{ll}F_1(t),&-1\leq t\leq t^*\\[2mm]
F_2(t),&~~t^*\leq t\leq 1\end{array}\right.$$
where $F_1(t^*)=F_2(t^*)$, but $F_1'(t^*)\neq F_2'(t^*)$.
The result obtained from the first integration by parts in the proof of 
Theorem \ref{int2-thm} still holds since $F$ is continuous. However, the
resulting integral must be split into two and so we have
\begin{eqnarray*}
\int_{-1}^1F(t)e_n(t)~dt&=&-\frac{1}{(n+2)!}\left[\int_{-1}^{t^*}F_1'(t)
\nu^{(n+2)}(\kappa(t))\phi_{n+2}(t)~dt+\right.\\
&&\hspace{22mm}\left.\int_{t^*}^1F_2'(t)
\nu^{(n+2)}(\kappa(t))\phi_{n+2}(t)~dt\right].
\end{eqnarray*}
For $n$ odd, the next step in the proof of Theorem \ref{int2-thm} was to 
integrate by parts
again, with the boundary terms vanishing. However, in this case, the
boundary terms do not vanish, but integration by parts gives
\begin{eqnarray*}
\int_{-1}^1F(t)e_n(t)~dt&=&-\frac{1}{(n+2)!}\left(F_1'(t^*)-F_2'(t^*)\right)
\nu^{(n+2)}(\kappa(t^*))\psi_{n+3}(t^*)\\
&&+\frac{1}{(n+2)!}\left[\int_{-1}^{t^*}\frac{d}{dt}\left(F_1'(t)
\nu^{(n+2)}(\kappa(t))\right)\psi_{n+3}(t)~dt+\right.\\
&&\hspace{22mm}
\left.\int_{t^*}^{1}\frac{d}{dt}\left(F_2'(t)
\nu^{(n+2)}(\kappa(t))\right)\psi_{n+3}(t)~dt\right].
\end{eqnarray*}
Now $\psi_{n+3}(t)$ is of one sign \cite{RR78} and so, by the Mean Value
Theorem for Integrals, the derivative term can be taken outside of each
integral, giving two terms each of which are similar to that obtained in 
Theorem \ref{int2-thm}. Converting back to $x$ coordinates, as in the proof of
Theorem \ref{thm2}, we again find that the integral terms are
$O(h^{n+4})$. However, the boundary terms, which were not present
previously, are $O(h^{n+3})$. Now when we sum over all the intervals, we
sum the integrals and this sum gives an error of $O(h^{n+3})$ since
$N=O(h^{-1})$, as previously. In the case of the boundary terms, we note
that only intervals containing a discontinuity of the derivative of $f$
have such boundary terms, and this is a fixed number which does not depend
on $h$. Thus, the contribution of these boundary terms is also $O(h^{n+3})$
and so the total error is $O(h^{n+3})$ as previously.

A similar analysis applies for $n$ even.
\end{proof}

\begin{table}[t]
\begin{center}
\begin{tabular}{|c|c|}
\hline
\hspace{2mm}$n$\hspace{2mm} & \hspace{1mm}Error Term\hspace{1mm} \\
\hline 
0 & $O(h^2)$ \\
1 & $O(h^4)$ \\
2 & $O(h^4)$ \\
3 & $O(h^6)$ \\
4 & $O(h^6)$ \\
\hline
\end{tabular}
\end{center}
\hspace{5mm}
\caption{The order of the error term in \Ref{err1} and \Ref{err2} for 
different values of $n$.}
\label{errtable}
\end{table}

A summary of the order of errors for different values of $n$ is given in
Table \ref{errtable}, which shows a similar pattern to the error in
Newton-Cotes integration.

\subsection{Gaussian quadrature}

In Case I, the integral on the right hand side of \Ref{res1} is that of a
polynomial, and so can be evaluated exactly. However, in Case II and for a
general function $F(t)$, the first integral on the right hand side of
\Ref{integral2} may have to be evaluated numerically using Gaussian
quadrature. When integrating a function $G(t)$ on the 
interval $[-1,1]$ using an $m$ point Gaussian quadrature rule, the error 
term is given by
$$\frac{G^{(2m)}(\eta)}{(2m)!}\int_{-1}^1 [P_m(t)]^2~dt,$$
where $P_m(t)$ is the $m^{\rm th}$ Legendre polynomial and $\eta\in(-1,1)$
\cite{BF10}. Converting back to the original $x$ coordinates, we note again
that each derivative of a function introduces a power of $h$, since
$dx/dt=\frac{1}{2}(n+1)h$. The functions we integrate always
involve the density function $D(t)$ and the change of function from $D(t)$
(which is approximated by $p_n(t)$) back to $d(x)$ introduces another power of
$h$. Thus, the error associated with a simple $m$ point Gaussian quadrature
rule is $O(h^{2m+1})$.

In Case II, as shown in the proof of Theorem \ref{thm2}, the approximation 
error for a single integral for $n$ odd is $O(h^{n+4})$. Thus, the error 
from the quadrature will be the same order as the approximation error if
$m=(n+3)/2$. Similarly for $n$ even, the approximation error for a single
integral is $O(h^{n+3})$
and so the quadrature error will be the same order if $m=(n+2)/2$.
Thus, for $n=2k-1$ or $n=2k$, the number of integration points used should 
be at least $m=k+1$ to ensure that the error in the integration does not 
dominate the approximation error.

\section{Computing the Polynomial Density from the \\
Frobenius-Perron Equation}
\label{method}

\subsection{Discretisation of the Frobenius-Perron equation}\label{FP-sec}

Using the approach described in the previous section, we have reduced
the description of the density associated with an iteration function 
$g:I\to I$ to a finite dimensional approximation
which is characterised by the $N$-dimensional vector $\m$ whose
components are 
$$m_i=\int_{I_i}d\mu,\quad i=1,2,\ldots,N.$$
To determine this vector, we must solve the Frobenius-Perron equation 
given by
$$\int_J d\mu=\int_{g^{-1}(J)} d\mu$$
for any interval $J\subset I$. To obtain a set of determining
equations for the vector $\m$, we use the Frobenius-Perron equation
with $J=I_i$, $i=1,2,\ldots,N$ giving
$$\int_{I_i}d\mu=\int_{g^{-1}(I_i)}d\mu,\quad i=1,2,\ldots,N.$$
Now by definition, $\int_{I_i}d\mu=m_i$. Similarly, the right hand
side can be written as
$$\int_{g^{-1}(I_i)}d\mu=\bigcup_{k=1}^K\int_{x_{L,k}}^{x_{R,k}}d(x)~dx,$$
where $K$ is the number of preimages of the interval $I_i$, and
for each $k=1,\ldots,K$, $g(x_{L,k})=x_{i-1}$, $g(x_{R,k})=x_i$.
Now each of the preimages of the interval $I_i$ may be contained within a
single interval $I_j$ for some $j$ or it may contain some whole intervals
and only parts of others. The evaluation of this type of integral was 
considered in Case I in the previous section.

We note that when $n=0$ (piecewise constant polynomial approximations),
this is the standard Ulam method for approximating the invariant
density \cite{Ul60a}.

\subsection{Convergence}

We now consider the convergence of this method as $N\to\infty$.
We define the approximation space
\[
\Delta_N^n = \left\{ f:[0,1]\to [0,1] : f|_{[x_i,x_{i+n+1}]} \in \Pi_n, 
i=0,n+1,2(n+1),\ldots,N-(n+1)\right\},
\]
of functions which are piecewise polynomials of degree $n$ on $n+1$
adjacent intervals from the given partition.  Note that $\dim(\Delta_N^n) =
N$ and functions in $\Delta_N^n$ are not necessarily continuous.

In Section \ref{poly-sec}, for a given density $d$, we constructed a 
function $p_N^n\in\Delta_N^n$ which satisfies
\[
\int_{I_i} p_N^n(x)\; dx = m_i = \int_{I_i} d(x)\; dx, \quad i=1,\ldots, N.
\]
In other words, $p_N^n=L_N^n d$, where the projection 
$L_N^n:L^1([0,1],\R)\to\Delta_N^n$ is characterised by
\begin{equation}\label{eq:proj}
\int_{I_i} (L_N^n d)(x)\; dx = \int_{I_i} d(x)\; dx, \quad i=1,\ldots, N,
\end{equation}
for $d\in L^1([0,1],\R)$.  Let $(\;,\;)$ denote the duality pair between
$L^1([0,1],\R)$ and its dual and let $T_N$ be the space spanned by the
characteristic functions $\chi_1,\ldots,\chi_N$ on the intervals
$I_1,\ldots,I_N$. Then (\ref{eq:proj}) can alternatively be written as
\begin{equation}\label{eq:proj2}
(L_N^nd,v)=(d,v) \quad \mbox{for all } v \in T_N.
\end{equation}
Thus, $L_N^n d$ is the \emph{Petrov-Galerkin projection} of $d$ (with
respect to the spaces $\Delta_N^n$ and $T_N$).

In Section \ref{FP-sec}, for a given map $g:[0,1]\to [0,1]$, we constructed 
an approximate invariant density $\tilde d\in\Delta_N^n$ by requiring that
\begin{equation}\label{eq:discrete_FP}
\int_{I_i} \tilde d(x)\; dx = \int_{g^{-1}(I_i)} \tilde d(x)\; dx,
\quad i=1,\ldots,N.
\end{equation}
%(Note that the left hand side of the equation equals $m_i$.) [This is not
%true if $\tilde d$ is a solution of the equation. It is only an
%approximation to $m_i$.
The next Lemma shows that this is equivalent to computing $\tilde d\in\Delta_N^n$ such that
$$L_N^nP\tilde d = \tilde d,$$
where $P$ is the Frobenius-Perron operator, i.e.\
$\tilde d$ is a fixed point of the \emph{discretised Frobenius-Perron 
operator}
\begin{equation}\label{dfpo}
P_N^n := L_N^n P.
\end{equation}

\begin{lemma}
A function $\tilde d\in\Delta_N^n$ satisfies (\ref{eq:discrete_FP}) if and only
if it is a fixed point of the discrete Frobenius-Perron operator $P_N^n$.
\end{lemma}

\begin{proof}
If $\tilde d$ satisfies (\ref{eq:discrete_FP}) then
$(\tilde d,v) = (P\tilde d,v)$ for all $v\in T_N$ since the characteristic 
functions on the intervals $I_i$ are a basis of $T_N$. By the definition of the 
projection $L_N^n$ in \Ref{eq:proj2}, $(L_N^nP\tilde d,v) = (P\tilde d,v)$ 
for all $v\in T_N$. Combining these two equalities, we arrive at
\[
(L_N^nP\tilde d,v) = (\tilde d,v)\quad\mbox{for all } v\in T_N,
\]
and since $\tilde d\in\Delta_N^n$ and the projection is unique, it follows that $P_N^n\tilde d=\tilde d$.

To prove the converse, let $\tilde d\in\Delta_N^n$ be a fixed point of
$P_N^n$, then $(L_N^nP\tilde d,v)=(\tilde d,v)$ for all $v\in T_N$. From
the definition of $L_N^n$, $(L_N^nP\tilde d,v)=(P\tilde d,v)$ for all $v\in
T_N$.  Combining these two equalities, we obtain
\[
(P\tilde d, v)=(\tilde d,v)\quad\mbox{for all }v\in T_N,
\]
which is equivalent to (\ref{eq:discrete_FP}).
\end{proof}

Analogous to the proof of Lemma 8 in \cite{DiDuLi93a} we can show:
\begin{lemma}\label{FP-fixed-point}
The discretised Frobenius-Perron operator $P_N^n$ has a nonzero fixed point
\[
P_N^nd_N^n = d_N^n\in\Delta_N^n,
\]
which satisfies
\begin{equation}\label{norm1}
\int_0^1 d_N^n(x)~dx=1.
\end{equation}
\end{lemma}

We use the framework of \cite{DiDuLi93a} in order to prove convergence of our
scheme in the case that $g$ is piecewise $C^2[0,1]$ and stretching (i.e.\
$\inf_{x\in [0,1]} |g'(x)| > 1$).  The only difference to the setup in
\cite{DiDuLi93a} is that we are dealing with a Petrov-Galerkin projection instead of a standard Galerkin projection.  We are thus working with a fixed conjugate basis $\{A^i\}_0^n$ of $\Delta_{n+1}^n$,  
\[
A^i = \chi_{i+1}, \quad \quad i=0,\ldots,n,
\]
and choose the basis $\{a^i\}_0^n$ of $\Delta_{n+1}^n$ such that
\[
(a^i,A^k) = \delta_{ik}, \quad i,k=0,\ldots,n.
\]
Lemma 10 of \cite{DiDuLi93a} now yields stability of the projection $L_N^n$:
\begin{lemma} There exist constants $\gamma_1(n)$ and $\gamma_2(n)$ such that for $d\in L^1(0,1)$
\[
\| L_N^n d\|_1 \leq \gamma_1(n) \|d\|_1
\]
and if $d\in BV(0,1)$ then
\begin{equation}\label{gamma2}
\bigvee_0^1 L_N^n d \leq \gamma_2(n) \bigvee_0^1 d.
\end{equation}
\end{lemma}
As in \cite{DiDuLi93a}, consistency of the projection is a standard result from approximation theory:
\begin{lemma}
Let $d\in L^1(0,1)$, then
\[
\lim_{N\to\infty} \|d-L_N^n d\|_1 = 0.
\]
\end{lemma}

If $g$ is piecewise $C^2[0,1]$ and stretching, then the Lasota-Yorke
theorem \cite{LaYo73a} implies that for any function $d\in L^1(0,1)$ of
bounded variation, $Pd$ is of bounded variation as well and there exist
constants $\alpha>0$ and $\beta \leq 2/M$ (with $M=\inf_{x\in [0,1]}
|g'(x)|>1$) such that
\begin{equation}\label{LY}
\bigvee_0^1 Pd \leq \alpha \|d\|_1 + \beta \bigvee_0^1 d.
\end{equation}
Corollary 3 and 4 of \cite{DiDuLi93a} now yield convergence of our scheme:
\begin{theorem}\label{conv-thm}
Let $g$ be piecewise $C^2[0,1]$ and stretching.  If $\beta <1$ and 
$\gamma_2(n)\beta < 1$ then the sequence $(d_N^n)_N$ of fixed points of the 
discretised Frobenius-Perron operator has a subsequence which converges in 
$L^1(0,1)$ to an invariant density $d$ of $g$. Moreover,
\[
\|d-d^n_N\|_1 = O(\|d-L_N^n d\|_1) \quad \mbox{as }N\to\infty.
\]
\end{theorem}

\begin{corollary}\label{corol}
Under the conditions of Theorem \ref{conv-thm} and for a given value of
$n$,
\begin{equation}\label{den-err}
\|d-d_N^n\|_1=O(h^{n+1})=
O\left(\frac{1}{N^{n+1}}\right)\quad \mbox{as }N\to\infty.
\end{equation}
\end{corollary}

\begin{proof}
The function $L_N^nd$ is the piecewise polynomial approximation to $d$
assuming that the measures $m_i$ on the intervals $I_i$, $i=1,\ldots,n$ are
known exactly. This is precisely the situation that we considered in
Theorem \ref{case1thm}. Thus,
\begin{eqnarray*}
\|d-L_N^nd\|_1&=&\int_I|d(x)-(L_N^nd)(x)|~dx\\
&=&\sum_{i=1}^N\int_{I_i}|d(x)-(L_N^nd)(x)|~dx\\
&=&\sum_{i=1}^N\int_{t_\ell}^{t_r}|D_i(t)-p_{n,i}(t)|~dt,
\end{eqnarray*}
for appropriate values of $t_{\ell}$ and $t_r$ and where $D_i(t)$ and
$p_{n,i}(t)$ are obtained from the restriction of $d(x)$ and $L_N^nd(x)$
respectively to the interval $I_i$. By \Ref{res2}, the final
integral above is $O(h^{n+2})$ and since a sum of $N=O(h^{-1})$ such terms
is required, then
$$\|d-L_N^nd\|_1=O(h^{n+1}).$$
Finally, Theorem \ref{conv-thm} tells us that
$$\|d-d_N^n\|_1=O(\|d-L_N^nd\|_1)=O(h^{n+1})=O\left(\frac{1}{N^{n+1}}
\right),$$
as required.
\end{proof}

\begin{remark}\label{LaYo-conds}
Theorem \ref{conv-thm} requires the map $g$ to be stretching, i.e.\
$\inf_{x\in[0,1]}|g'(x)|>1$ as this is a condition of the Lasota-Yorke 
Theorem \cite{LaYo73a}. However, by Theorem 3 of \cite{LaYo73a}, this
condition can be relaxed by requiring that $g^m$ be stretching for some
positive integer $m$ and that $\inf_{x\in[0,1]}|g'(x)|>0$. Under these
conditions, Theorem \ref{conv-thm} also holds. 
\end{remark}

\subsection{Examples}\label{examp}

As an example of the preceding theory, we consider in detail the map 
$g_1:[0,1]\to[0,1]$ defined by
\begin{equation}\label{gdef}
g_1(x)=\left\{\begin{array}{ll}
\displaystyle \frac{2x}{1-x^2}, & 0\leq x\leq\sqrt{2}-1\\[3mm]
\displaystyle \frac{1-x^2}{2x}, & \sqrt{2}-1\leq x\leq 1.
\end{array}\right.
\end{equation}
The invariant density for this map (which is the map $S_3$ of \cite{DiDuLi93a})
is
$$d_1(x)=\frac{4}{\pi(1+x^2)}.$$
The function $g_1$ and the density $d_1$ are shown in Fig.\ \ref{fig1}.

We note that $g_1$ is piecewise $C^2[0,1]$, but is not stretching 
as $g_1'(1)=-1$. However, $g_1^2$ is stretching and 
$\inf_{x\in[0,1]}|g_1'(x)|=1$
and so the conditions of Remark \ref{LaYo-conds} are satisfied.

\begin{figure}[t]
\centerline{
\includegraphics[width=0.45\textwidth]{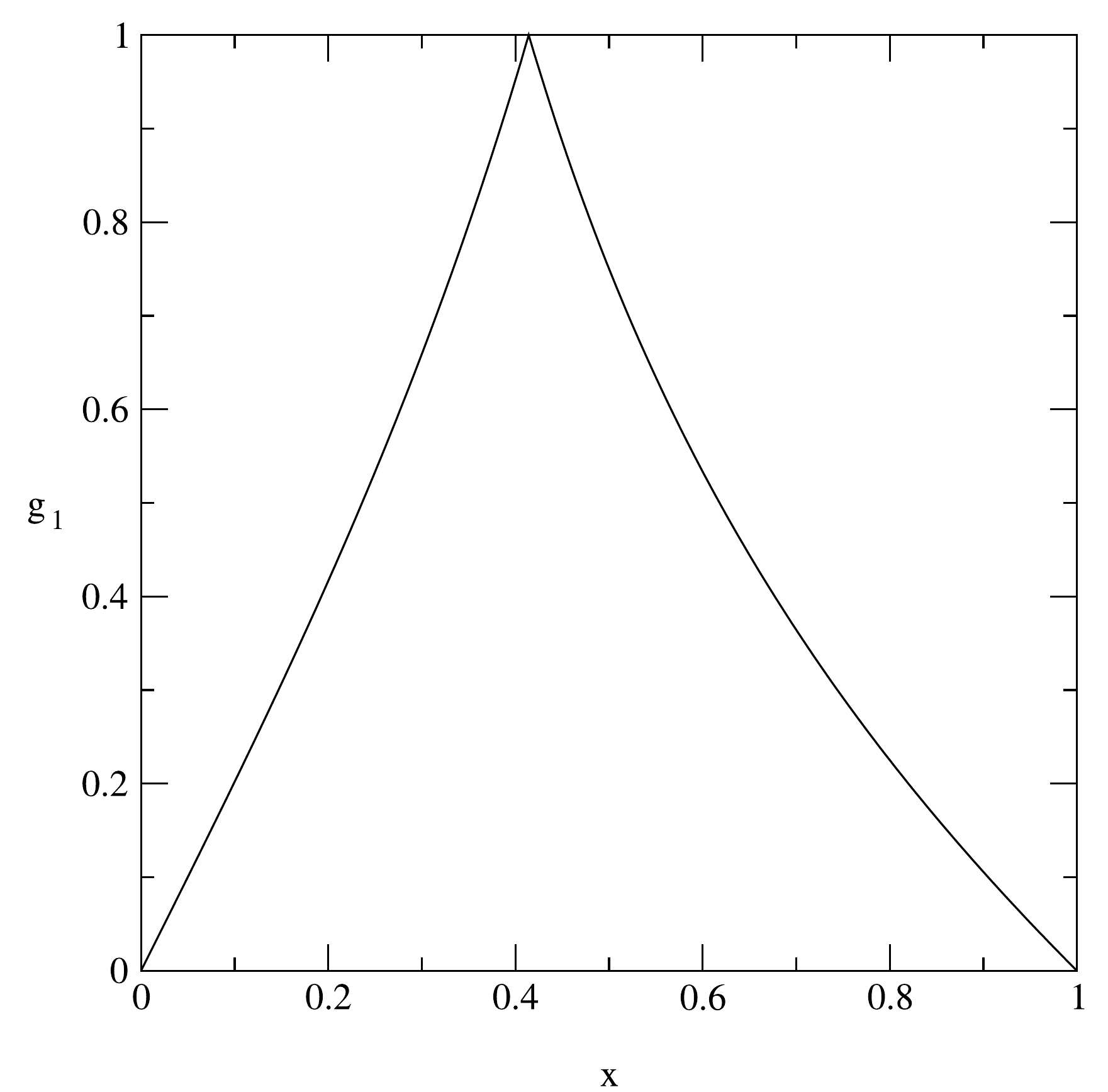}
\includegraphics[width=0.45\textwidth]{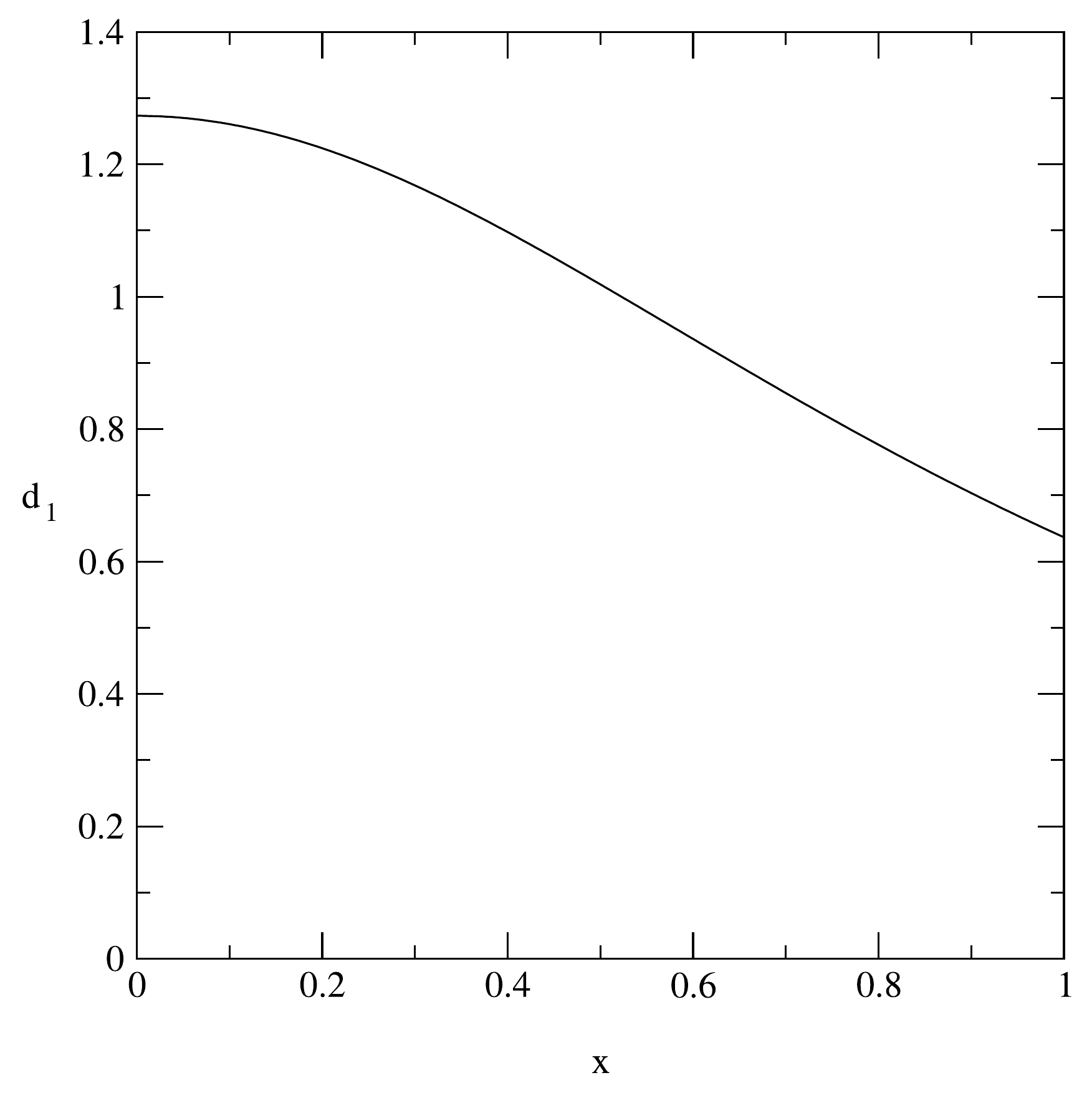}}
\caption{The map $g_1$ defined by \Ref{gdef} and it's density function $d_1$.}
\label{fig1}
\end{figure}

We divided the interval $[0,1]$ into $N$ subintervals of width $h=1/N$ 
and took groups of $n+1$ subintervals to construct a measure-preserving
polynomial approximation to the density on these intervals.
The method described in the previous section was used to compute the
matrix representation $A$ of the Frobenius-Perron operator and the
eigenvector $\tilde\m$ of $A$ associated with the eigenvalue nearest to 
$+1$ was computed using the inverse power method \cite{BF10}. Since the true
density function is known for this example, the one norm of the difference 
between the exact function $d_1(x)$ and the piecewise polynomial approximation 
$d_N^n(x)$ was determined. The program was written in {\sc Maple} so 
that the accuracy of the calculations could be increased if required.
The results are shown in Fig.\ \ref{fig:density_g1}. 
%The final column contains 
%the quantity $q_N$ which is defined by
%\begin{equation}\label{qn}
%q_N={\log\left({||d_1(x)-d_{N/2}^n(x)||_1\over
%||d_1(x)-d_N^n(x)||_1}\right)/\log 2}.
%\end{equation}
%If $||d_1(x)-d_N^n(x)||_1=O(h^q)$ then $\lim_{N\to\infty}q_N=q$. 
We note that the slope of the final segment of each of these lines is given
for $n=0,\,1,\,2,\,3$ by $-1.035$, $-2.049$, $-3.049$ and $-3.917$ respectively,
which gives experimental verification of the theoretical result of
Corollary \ref{corol} that the rate of convergence is $O(h^{n+1})$.

%\begin{table}[t]
%\begin{center}
%\begin{tabular}{|c|c|c|c|}
%\hline
%$n$ & $N$ & $||d_1(x)-d_N^n(x)||_1$ & $q_N$ \\
%\hline
%~~0~~ &  4 & ~~$5.316891\times 10^{-2}$~~ & \\
%  &  8 & $2.379607\times 10^{-2}$ &  ~~1.159859~~ \\ 
%  & 16 & $1.168727\times 10^{-2}$ &  1.025786 \\ 
%  & 32 & $5.541273\times 10^{-3}$ &  1.076648 \\ 
%  & 64 & $2.678822\times 10^{-3}$ &  1.048619 \\ 
%  &~~128~~ & $1.307149\times 10^{-3}$ &  1.035175 \\ 
%\hline
%1 &  4 & $1.142647\times 10^{-2}$ & \\
%  &  8 & $2.443353\times 10^{-3}$ &  2.225445 \\ 
%  & 16 & $6.468501\times 10^{-4}$ &  1.917359 \\ 
%  & 32 & $1.666556\times 10^{-4}$ &  1.956562 \\ 
%  & 64 & $4.477744\times 10^{-5}$ &  1.896026 \\ 
%  &128 & $1.081713\times 10^{-5}$ &  2.049454 \\ 
%\hline
%2 &  6 & $1.213178\times 10^{-3}$ & \\
%  & 12 & $2.453729\times 10^{-4}$ &  2.305743 \\ 
%  & 24 & $2.145008\times 10^{-5}$ &  3.515921 \\ 
%  & 48 & $3.103123\times 10^{-6}$ &  2.789190 \\ 
%  & 96 & $3.403574\times 10^{-7}$ &  3.188598 \\ 
%  &192 & $4.113697\times 10^{-8}$ &  3.048543 \\ 
%\hline
%3 &  8 & $1.088299\times 10^{-4}$ &  \\ 
%  & 16 & $8.776219\times 10^{-6}$ &  3.632331 \\ 
%  & 32 & $7.390182\times 10^{-7}$ &  3.569918 \\ 
%  & 64 & $4.405983\times 10^{-8}$ &  4.068074 \\ 
%  &128 & $2.917823\times 10^{-9}$ &  3.916500 \\
%\hline
%\end{tabular}
%\end{center}
%\caption{Errors in computed invariant measures for different values of $n$
%and $N$ for the map $g_1$, where $q_N$ is given in \Ref{qn}.}
%\label{table1}
%\end{table}

\begin{figure}[t]
\centering
\includegraphics[width=0.5\textwidth]{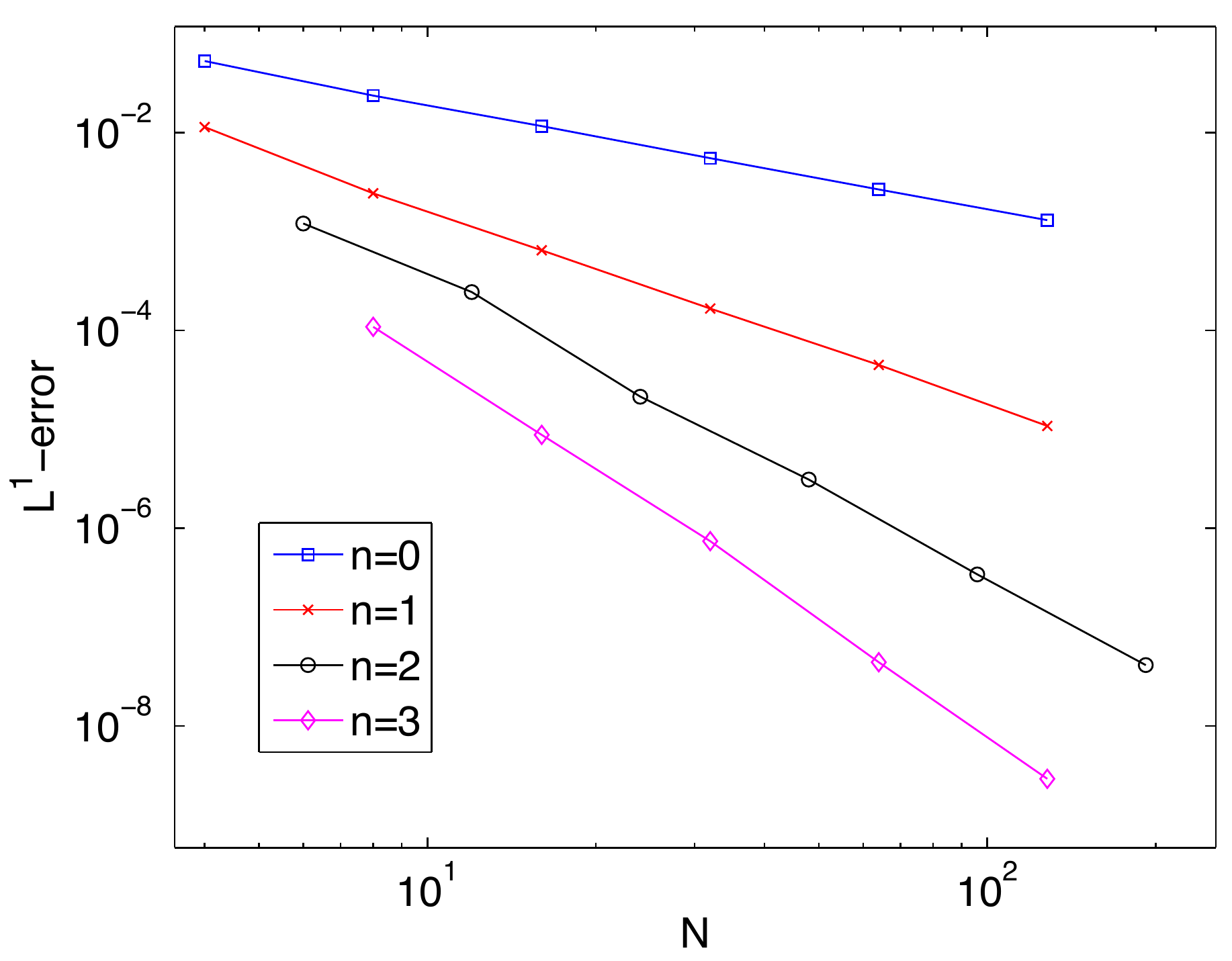}
\caption{$L^1$-errors in computed invariant densities for different values 
of $n$ and $N$ for the map $g_1$.}
\label{fig:density_g1}
\end{figure}

To illustrate the power of this method, taking 16 subintervals with
groups of 4 subintervals used to compute four cubic polynomial
approximations to the density gives an error of
$$||d_1(x)-d_{16}^3(x)||_1=8.776219\times 10^{-6}.$$
Thus, by using higher order polynomials, very accurate results can be
obtained using only a small number of subintervals. The standard Ulam
method using only piecewise constant approximations to the density and
the same number of intervals gives an error of
$$||d_1(x)-d_{16}^0(x)||_1=1.168727\times 10^{-2}.$$
For piecewise constant approximations, doubling $N$ approximately halves
the error, as can be seen in Figure~\ref{fig:density_g1}. Taking
the error for $N=128$ and successively halving it, we obtain an
approximation to the error for $N=128\times 2^7=16,384$ of $1.307149\times
10^{-3}/2^7=1.02121\times 10^{-5}$, which is still slightly larger than the
error above using $p_{16}^3(x)$!

We note that the results of \cite{DiDuLi93a} for this map appear to show
convergence rates of $O(h)$, $O(h^2)$ and $O(h^2)$ for the piecewise
constant, linear and quadratic cases. We get the same rates for the
constant and linear cases, but we have $O(h^3)$ for the piecewise quadratic
case.

We have also applied our method to the map
\begin{equation}\label{g2}
g_2(x)=\left\{\begin{array}{ll}
\displaystyle \frac{2x}{1-x}, & 0\leq x\leq\frac{1}{3}\\[3mm]
\displaystyle \frac{1-x}{2x}, & \frac{1}{3}\leq x\leq 1,
\end{array}\right.
\end{equation}
which is the map $S_4$ of \cite{DiDuLi93a}. This map has invariant density
given by
$$d_2(x)=\frac{2}{(1+x)^2}.$$
Our results for this map are similar to those for the previous map and
show the same rates of convergence.

Finally, we consider
$$g_3(x)=\left(\frac{1}{8}-2\left|x-\frac{1}{2}\right|^3\right)^{1/3}
+\frac{1}{2},$$
which is the map $S_2$ of \cite{DiDuLi93a}. In this case, the invariant density
is
$$d_3(x)=12\left(x-\frac{1}{2}\right)^2,$$
which is quadratic. From the proof of Theorem \ref{case1thm}, it is clear
that the
error term for the integrals involved in setting up the Frobenius-Perron
operator depend on the $(n+1)^{\rm th}$ derivative of the density. Since
the density is a quadratic function in this case, we would expect the error
to be zero with a piecewise quadratic approximation and this is indeed the
case. This is in contrast to the results of \cite{DiDuLi93a}, whose results
appear to show $O(h^3)$ convergence with piecewise quadratic approximations
for this map. We again achieve convergence rates of $O(h)$ and $O(h^2)$ for
piecewise constant and linear approximations.

The interesting aspect of this example is that the function $g_3$ does not
fulfil the conditions of Theorem \ref{conv-thm}, as it is neither $C^2[0,1]$
(as $g_3'(x)=\infty$ at the two points $x=\frac{1}{2}\pm 2^{-4/3}$) nor 
stretching (since the first and second derivatives at $x=1/2$ are zero). 
Also no iterates of $g_3$ are stretching and so Remark 
\ref{LaYo-conds} is no help either in this case. Thus,
it seems that the method converges for a wider class of maps than those
specified in Theorem \ref{conv-thm} and Remark \ref{LaYo-conds}. 
In fact, recently much progress has been made in establishing the 
existence of invariant densities for general interval maps which are 
not expanding, see for example \cite{ArLuVi09a,BrLuVa03a,BrRiSh08a}. 
However, it remains to be explored whether the theory 
developed in these papers yields the tools in order to prove Ulam's 
conjecture or the convergence of the scheme developed in this paper.

\section{Computing the Lyapunov Exponent Using Integration}

Having obtained a good approximation to the invariant density, we now want 
to compute an approximation to the Lyapunov exponent for the map $g$ which 
is given by
$$\sigma=\int_0^1\log|g'(x)|d(x)~dx.$$
Let $d_N^n$ be the discontinuous piecewise polynomial approximation to the
invariant density which is a nonzero fixed point of the discretised
Frobenius-Perron operator $P_N^n$ given by \Ref{dfpo}.
%Let $p(x)$ be the discontinuous piecewise polynomial approximation to the
%density with measures $m_i$ on the intervals $I_i$, $i=1,\ldots,N$ that are
%exact. Also, let $\tilde p(x)$ be the discontinuous piecewise polynomial
%approximation to the density with measures $\tilde m_i$ on the intervals
%$I_i$, $i=1,\ldots,N$ that are found by solving the discretised
%Frobenius-Perron equation \Ref{eq:discrete_fixed_point}. 
Then the approximation to the Lyapunov exponent that we can actually compute 
(ignoring numerical integration errors) is given by
$$\sigma_N^n=\int_0^1\log|g'(x)|d_N^n(x)~dx.$$
A simple error analysis gives that
\begin{eqnarray*}
|\sigma-\sigma_N^n|&=&\left|\int_0^1\log|g'(x)|(d(x)-d_N^n(x))~dx\right|\\
&\leq&\int_0^1\left|\log|g'(x)|(d(x)-d_N^n(x))\right|~dx\\
&\leq&||\log|g'|\;||_\infty||d-d_N^n||_1\\
&=&O(h^{n+1}),
\end{eqnarray*}
using H\"older's inequality and \Ref{den-err}, assuming that
$\log|g'|$ is bounded. Thus, it would appear that the error in
the Lyapunov exponent is determined by the error in the invariant density.
However, we now show that better results than this can often be obtained.

\begin{theorem}\label{lethm}
Assume that $g$ is piecewise $C^3[0,1]$ and stretching with $\beta < 1$ (cf.
\Ref{LY}) and $\gamma_2(n)\beta<1$ (cf. \Ref{gamma2}), and that 
$|g'|\in C^0[0,1]$. If the Frobenius-Perron operator $P$ has a unique invariant
density $d\in C^{n+2}[0,1]$ ($n$ odd) or $d\in C^{n+1}[0,1]$ ($n$ even) then
$$|\sigma-\sigma_N^n|=O(h^{n+2})$$
as $N=1/h\to\infty$.
\end{theorem}

\begin{proof}
We consider two different approaches to proving this result.

For the first proof, we note from Corollary \ref{corol} that
$$||d-d_N^n||_1=O(h^{n+1})$$
and so we write
\begin{equation}\label{derror}
d(x)=d_N^n(x)+h^{n+1}\delta(x).
\end{equation}
where $\delta\in L^1(0,1)$. Integrating by parts then gives
\begin{eqnarray*}
|\sigma-\sigma_N^n|&=&\left|\int_0^1\log|g'(x)|(d(x)-d_N^n(x))~dx\right|\\
&=&\left[\log|g'(x)|\int_0^x d(x')-d_N^n(x')~dx' \right]_0^1\\
&&-h^{n+1}\int_0^1\frac{d}{dx}\log|g'(x)|\int_0^x\delta(x')~dx'~dx.
\end{eqnarray*}
Clearly, the boundary term evaluated at $x=0$ is zero and since
$$\int_0^1d(x)~dx=\int_0^1d_N^n(x)~dx=1,$$
by \Ref{norm1}, then the boundary term evaluated at $x=1$
is also zero. Using the same methods as in the proof of Theorem \ref{thm2},
this integration by parts results in the error increasing by one power of
$h$ and so the error is $O(h^{n+2}$) as claimed.

For the second proof, we let $p_N^n=L_N^nd$ be the projection of the true 
invariant density $d$ onto the approximation space $\Delta_N^n$, where the
projection $L_N^n$ is defined by \Ref{eq:proj}. Then
\begin{eqnarray*}
|\sigma-\sigma_N^n|&=&\left|\int_0^1\log|g'(x)|(d(x)-p_N^n(x))~dx\right.\\
&&\left.+\int_0^1\log|g'(x)|(p_N^n(x)-d_N^n(x))~dx \right|\\
&\leq&\left|\int_0^1\log|g'(x)|(d(x)-p_N^n(x))~dx\right| \\
&& +\left|\int_0^1\log|g'(x)|(p_N^n(x)-d_N^n(x))~dx\right|.
\end{eqnarray*}
The first integral is of the type that we considered in Section
\ref{case2}. The stated conditions on $g$ ensure that $\log|g'(x)|$
satisfies the conditions of Theorem \ref{thm3} and so this Theorem gives
$$\int_0^1\log|g'(x)|(d(x)-p_N^n(x))~dx=\left\{\begin{array}{ll}
O(h^{n+3}), & n{~\rm odd}\\[2mm]
O(h^{n+2}), & n{~\rm even} \end{array}\right.$$
For the second integral, we first note that
$$p_N^n-d_N^n=L_N^n(d-d_N^n)=h^{n+1}\delta_N^n,$$
using \Ref{derror}, where
$$\delta_N^n=L_N^n\delta.$$
Integrating the second integral by parts then gives
\begin{eqnarray*}
\int_0^1\log|g'(x)|(p_N^n(x)-d_N^n(x))~dx&=&
\left[\log|g'(x)|\int_0^x p_N^n(x')-d_N^n(x')~dx' \right]_0^1\\
&&-h^{n+1}\int_0^1\frac{d}{dx}\log|g'(x)|\int_0^x\delta_N^n(x')~dx'~dx.
\end{eqnarray*}
Again, the boundary term at $x=0$ is zero and since
$$\int_0^1p_N^n(x)~dx=\int_0^1d_N^n(x)~dx=1,$$
by \Ref{norm1}, then the boundary term evaluated at $x=1$ 
is also zero. Using the same approach now as for the first proof above, the
error has increased by one power of $h$ and so this error is $O(h^{n+2})$.

Combining these two integrals, asymptotically the largest error is from the
second integral and so the error in the Lyapunov exponent is $O(h^{n+2})$
as claimed.
\end{proof}

%[{\em Is there a general result that says if you integrate by parts and can
%eliminate the boundary terms, then the error goes up by a power of $h$?}]

\subsection{Example}

We consider again the map $g_1$ defined by \Ref{gdef} that we considered 
previously in Section \ref{examp}. We note that for this example, $|g_1'(x)|$
is continuous at the point $x=\sqrt{2}-1$, which is one of the conditions
of Theorem \ref{lethm}.

We used the method described in Section \ref{method} to obtain the
piecewise
polynomial approximation $d_N^n(x)$ to the density $d(x)$ and we then 
evaluated
$$\sigma_N^n=\int_0^1\log|g_1'(x)|d_N^n(x)~dx.$$
The calculations were done in {\sc Maple} and so the integration was
performed accurately by a {\sc Maple} routine rather than using Gaussian
quadrature.
 
Since we know the true density function for this map, we can also
accurately compute the true value of the Lyapunov exponent $\sigma$, which
is found numerically to have the value $\log 2$. 
The results of these computations are shown in Fig.\ \ref{fig:lyap_g1}.
%[{\em Is this known analytically?}] 
%We then computed 
%\begin{equation}\label{rn}
%r_{N}=\log\left(\frac{|\sigma_{N/2}^n-\sigma|}
%{|\sigma_N^n-\sigma|}\right)/\log 2
%\end{equation}
%and if $\sigma_N^n=\sigma+O(h^r)$, then $\lim_{N\to\infty}r_N=r$.
%The numerical results obtained are shown in Table \ref{table2}.

\begin{figure}
\centering
\includegraphics[width=0.5\textwidth]{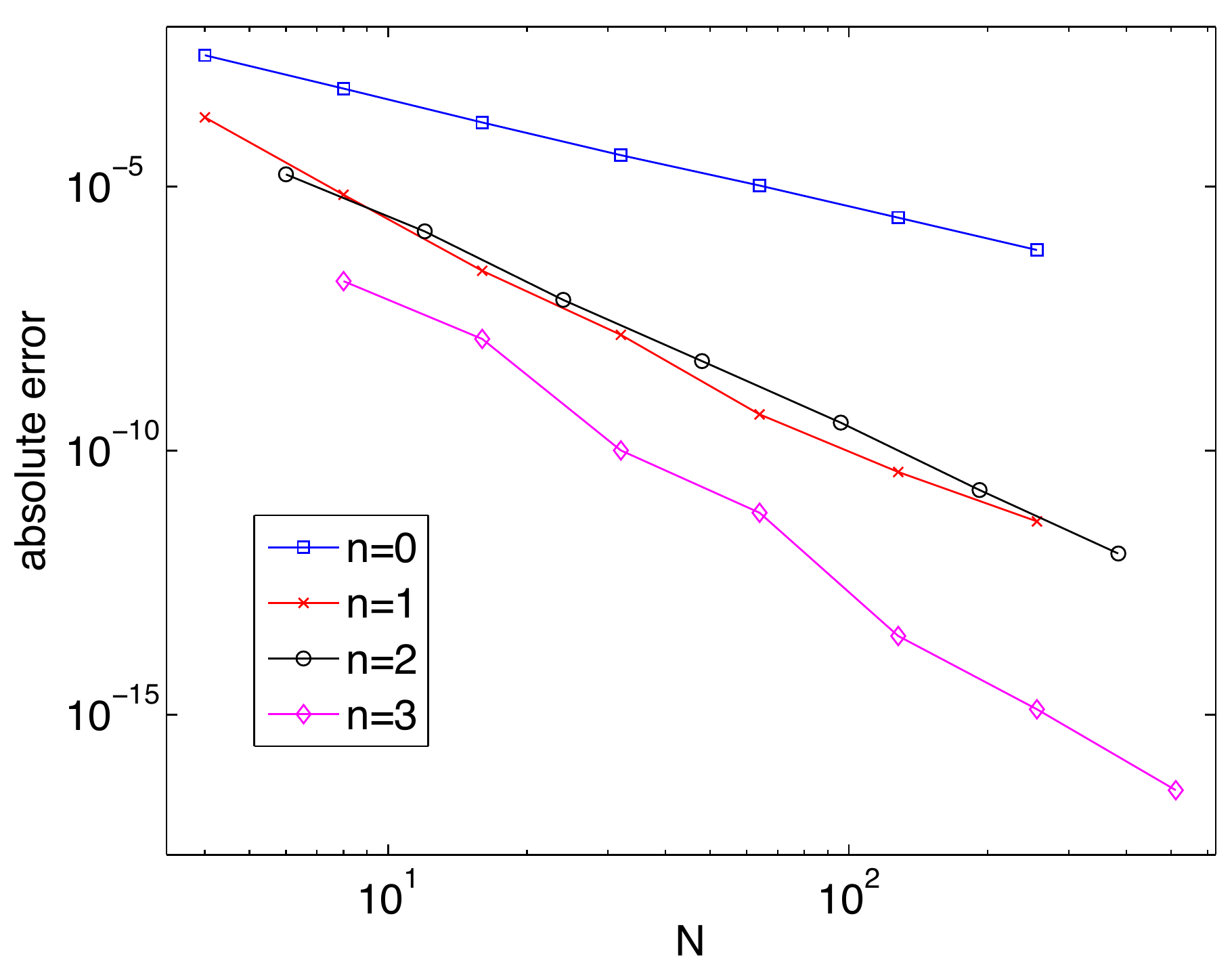}
\caption{Absolute errors in the computed values of the Lyapunov exponent
for the map $g_1(x)$.}
\label{fig:lyap_g1}
\end{figure}

The slope of the final segment of the lines for $n=0$ and $n=2$ are given
by $-2.031$ and $-4.001$ respectively, and so, for even values of $n$, 
the results agree well with the
prediction of Theorem \ref{lethm} that the rate of convergence is
$O(h^{n+2})$. However, for $n=1$, the rate of convergence initially is very
similar to that for $n=2$, which is $O(h^4)$, or equivalently $O(h^{n+3})$, 
and this is higher than that
predicted by Theorem \ref{lethm}. However, we note that the slope of the
final segment of the line for $n=1$ is $-3.105$, which is much closer to
the predicted rate of convergence of $O(h^3)$. Similarly, for $n=3$, the 
rate of
convergence estimate initially oscillates, but seems to be higher than the
prediction of $O(h^5)$. However, considering again the slope of the final 
segment of the line, we obtain a value of $-5.091$, which is close to the 
predicted rate of $O(h^5)$.

These results can be understood from the second proof of Theorem
\ref{lethm}. In that proof, the error was broken up into two terms. For $n$
even, the two separate terms were both $O(h^{n+2})$. However, for $n$ odd,
the terms were of different order so that approximately
$$|\sigma-\sigma_N^n|\approx c_1h^{n+3}+c_2h^{n+2}.$$
In this case, the asymptotic rate of convergence as $h\to 0$ is clearly
$O(h^{n+2})$. However, the rate of convergence that is observed for finite
values of $h$ also depends on the magnitude of the two constants $c_1$ and
$c_2$. If $c_1\gg c_2$, then for relatively ``large'' values of $h$, the
dominant term in the error will be the first one and so the rate of
convergence will appear to be $O(h^{n+3})$. However, as $h$ decreases,
eventually the higher power of $h$ will ensure that the first term becomes
smaller than the second, and then the true rate of convergence of
$O(h^{n+2})$ will be observed. Clearly, this is what is happening in our
example.

We note that even though the $n=1$ and $n=2$ cases have different
asymptotic rates of convergence, the actual magnitude of the errors shown
in Figure~\ref{fig:lyap_g1} are similar in these two cases. We have noted that for 
odd $n$, the first integral which is $O(h^{n+3})$ seems to dominate the 
errors for moderate values of $h$, and this implies that the error initially
decreases at a similar rate to that for the next even value of $n$. Thus,
even though theoretically the results for $n=2$ should be better than those
for $n=1$, and would
be better for sufficiently small $h$, for moderate values of $h$, the
results for these two values are comparable.

Similar results are obtained for the map $g_2(x)$ given by \Ref{g2}.

\section{Future Directions}

\subsection{Stochastic perturbations}

By construction, our approach requires the invariant density to be smooth.
For deterministic maps, this is not the generic situation. However, in
applications
one is often faced with a system which is additionally perturbed by (small)
random influences.  In these cases, instead of the deterministic system
$g$,
one considers the dynamics
$$x' = g(x) + \xi,$$
where $\xi$ is chosen from a given probability distribution $\mu$.  Suppose
that $\mu$ is absolutely continuous with density $h$.  Then the associated
Frobenius-Perron operator on $L^1$ is given by
\begin{equation}\label{eq:FP_noise}
Pf(x) = \int h(g(y)-x) f(y)\; dy, \quad f\in L^1.
\end{equation}
For common distributions
(like a normal distribution or the uniform distribution supported on a
small
ball  around $0$) the associated Markov chain possesses finitely many
invariant measures which, according to (\ref{eq:FP_noise}), inherit the
smoothness of the distribution \cite{DeJu99a}.

In this setting, the approach proposed in this
paper yields a highly accurate method for the approximation of the
invariant distribution. For very small perturbations and in the case of
a non-smooth invariant density, however, one will still need many modes
for this.  It would be interesting to investigate how the approximation
error behaves in dependence on the magnitude of the perturbation.

\subsection{Higher dimensions}

The methods described above for one-dimensional maps
can easily be generalised to higher dimensions. We briefly consider
only the case of a bilinear approximation to the density given the
total measure on four neighbouring squares in two dimensions. We assume 
that a change of
variables has been performed so that the region of interest is
$[-1,1]^2$. We take $t$ and $\tau$ as the two independent variables
and we want to approximate the density function $D(t,\tau)$ over this
region.

We first define $t_0=\tau_0=-1$, $t_1=\tau_1=0$ and $t_2=\tau_2=1$ and
assume that we know the four values
$$M_{ij}=\int_{\tau_{i-1}}^{\tau_i}\int_{t_{j-1}}^{t_j}D(t,\tau)~dtd\tau,
\quad i,j=1,2.$$
We then want to construct a bilinear approximation to the density
which preserves the total measure on each of the four subregions.
We write the polynomial approximation as
$$p_1(t,\tau)=\sum_{k,l=1}^2M_{kl}\ell_{1,k,l}(t,\tau).$$
The basis functions can be determined from the conditions
$$\int_{\tau_{i-1}}^{\tau_i}\int_{t_{j-1}}^{t_j}\ell_{1,k,l}(t,\tau)~dtd\tau
=\left\{\begin{array}{ll}0, & i\neq k~{\rm or}~j\neq l\\[3mm]
1, & i=k~{\rm and}~j=l \end{array}\right.\quad i,j=1,2.$$
These conditions give rise to the following basis functions:
\begin{eqnarray*}
\ell_{1,1,1}(t,\tau)&=&\frac{1}{4}-\frac{1}{2}t-\frac{1}{2}\tau+t\tau\\
\ell_{1,1,2}(t,\tau)&=&\frac{1}{4}+\frac{1}{2}t-\frac{1}{2}\tau-t\tau\\
\ell_{1,2,1}(t,\tau)&=&\frac{1}{4}-\frac{1}{2}t+\frac{1}{2}\tau-t\tau\\
\ell_{1,2,2}(t,\tau)&=&\frac{1}{4}+\frac{1}{2}t+\frac{1}{2}\tau+t\tau.
\end{eqnarray*}
This polynomial basis can then be used to approximate the invariant
density, again giving a discontinuous approximation over the whole region.
A similar error analysis as in the one-dimensional case can also be
performed.

\bibliographystyle{amsplain}

\end{document}